\RequirePackage{fix-cm}
\documentclass[smallextended]{svjour3}       % onecolumn (second format)
\smartqed  % flush right qed marks, e.g. at end of proof
\usepackage{graphicx}
 \usepackage{mathptmx}      % use Times fonts if available on your TeX system
\usepackage{latexsym}
 \usepackage{amsmath,amssymb}
\DeclareMathOperator{\pr}{\mathsf P}
\DeclareMathOperator{\M}{\mathsf E}
\DeclareMathOperator{\supp}{supp}
\DeclareMathOperator{\esssup}{ess\,sup}
\newcommand{\eps}{\varepsilon}

\newcommand{\R}{\mathbb{R}}
\newcommand{\F}{\mathcal{F}}
\newcommand{\LT}{L_2([0,1])}

 \journalname{Methodology and Computing in Applied Probability}
\begin{document}
%Удк 519.21

\title
{Approximation of fractional Brownian motion by martingales}


\titlerunning{Approximation of fBm by martingales}        % if too long for running head

\author{Sergiy Shklyar         \and
        Georgiy Shevchenko \and
        Yuliya Mishura \and
        Vadym Doroshenko \and
        Oksana Banna  %etc.
}

\authorrunning{S. Shklyar, G. Shevchenko, Yu. Mishura, V. Doroshenko, O. Banna} % if too long for running head

\institute{S. Shklyar,
 \at
              Kyiv National Taras Shevchenko University, Faculty of Mechanics and Mathematics,
Volodymyrska 64, 01601 Kyiv, Ukraine \\
              Tel.: +380-44-259-03-92\\
              Fax: +380-44-259-03-92\\
              \email{shklyar@univ.kiev.ua}           %  \\
           \and
           G. Shevchenko
 \at
              Kyiv National Taras Shevchenko University, Faculty of Mechanics and Mathematics,
Volodymyrska 64, 01601 Kyiv, Ukraine \\
              Tel.: +380-44-259-03-92\\
              Fax: +380-44-259-03-92\\
              \email{zhora@univ.kiev.ua}
              \and
              Yu. Mishura
 \at
              Kyiv National Taras Shevchenko University, Faculty of Mechanics and Mathematics,
Volodymyrska 64, 01601 Kyiv, Ukraine \\
              Tel.: +380-44-259-03-92\\
              Fax: +380-44-259-03-92\\
              \email{myus@univ.kiev.ua}           %  \\
           \and
           V. Doroshenko
 \at
              Kyiv National Taras Shevchenko University, Faculty of Mechanics and Mathematics,
Volodymyrska 64, 01601 Kyiv, Ukraine \\
              Tel.: +380-44-259-03-92\\
              Fax: +380-44-259-03-92\\
              \email{vadym.doroshenko@univ.kiev.ua}           %  \\
           \and
           O. Banna
 \at
              Kyiv National Taras Shevchenko University, Economics Faculty,
Volodymyrska 64, 01601 Kyiv, Ukraine \\
              Tel.: +380-44-259-03-92\\
              Fax: +380-44-259-03-92\\
              \email{okskot@ukr.net}           %  \\
}

\date{Received: date / Accepted: date}

\maketitle

\begin{abstract}
We study the problem of optimal approximation of a fractional Brownian motion by martingales.
We prove that there exist a unique martingale closest to fractional Brownian motion in a specific sense.
It shown that this martingale has a specific form. Numerical results concerning the approximation problem are given.
 %MSC2010
\end{abstract}

\maketitle
\section{Introduction}

Let $B^H=\{B_t^H, \mathcal{F}_t^{B^H}, t\in [0,1]\}$ be a fractional Brownian motion with Hurst index $H\in (0,1)$. It means that $B^H$ is a centered Gaussian process with a covariance function $\M[B_t^H B_s^H]=\frac12(s^{2H}+t^{2H}-|t-s|^{2H})$. It is well known that a fractional Brownian motion is neither a semimartingale nor a Markov process unless $H=1/2$. So a simple and natural question is how far is Brownian motion from being a martingale? That is, in a sense,  we look for the projection of fractional Brownian motion on the space of (square integrable) martingales. Thus, initially,  the problem is formulated in such a way: we are looking for a
square integrable $\mathcal{F}^{B^H}$-martingale $M$ that minimizes the value
$$d_H(M)^2:=
\sup_{t\in [0,1]}\M(B_t^H-M_t)^2.$$ To proceed with the solution of this problem, we can use the representation of the fractional Brownian motion via the standard Brownian motion  on the finite interval (\cite{Norros}).  Introduce the kernel $$K(t,s)= C_\alpha \Big( t^\alpha s^{-\alpha} (t-s)^\alpha-\alpha s^{-\alpha}\int_s^tu^{\alpha-1}(u-s)^{\alpha}du\Big)1_{0<s<t\leq 1},$$ where $C_\alpha = \alpha \left({\frac{(2\alpha+1)\Gamma(1-\alpha)}{\Gamma(\alpha+1)\Gamma(1-2\alpha)}}\right)^{1/2}$,
 $\Gamma$  is the Gamma function, $\alpha=H-1/2$. Then there exists  $\mathcal{F}^{B^H}$-Wiener process $W=\{W_t, \mathcal{F}_t^{B^H}, t\in [0,1]\}$ such that $B^H$ admits the representation
% \begin{equation}
\begin{gather}\label{fBm-repres}B_t^H=\int_0^1 K(t,s)dW_s=\int_0^t K(t,s)dW_s\\\nonumber =C_\alpha\int_0^t \Big( t^\alpha s^{-\alpha} (t-s)^\alpha-\alpha s^{-\alpha}\int_s^tu^{\alpha-1}(u-s)^{\alpha}du\Big)dW_s.\end{gather}
%\end{equation}
In what follows we consider fractional Brownian motion with $H\in (1/2,1)$, and in this case the kernel  $K(t,s)$ has a simpler form:
   \begin{equation}\label{oj4}K(t,s)=C_\alpha s^{-\alpha}\int_s^tu^\alpha(u-s)^{\alpha-1}du1_{0<s<t\leq 1}.\end{equation}

Turning back to our problem, we observe first that $B^H$ and $W$ generate the same filtration, so any square integrable $\mathcal{F}^{B^H}$-martingale $M$ admits a representation
\begin{equation}\label{itorep}
M_t = \int_0^t \alpha_s dW_s,
\end{equation}
where $\alpha$ is an $\mathcal{F}^{B^H}$-adapted square integrable process. Hence we can write
\begin{gather*}
\M(B_t^H - M_t)^2 = \M\left(\int_0^t (K(t,s)-\alpha_s) dW_s\right)^2 = \int_0^t \M(K(t,s)-\alpha_s)^2 ds\\
  = \int_0^t (K(t,s)- \M \alpha_s)^2 ds + \int_0^t \mathrm{Var}(\alpha_s) ds.
\end{gather*}
Consequently, it is enough to minimize $d_H(M)$ over \emph{Gaussian} martingales, i.e. those having representation
\eqref{itorep} with a non-random $\alpha$.

So, the main problem reduces to the following one:

\begin{itemize}
\item[(A)]Find
$$\inf_{a\in  L_2([0,1])}\sup_{t\in[0,1]}  \int_0^t (K(t,s)- a(s) )^2 ds $$
and a minimizing element $a\in  L_2([0,1])$ if the infimum is attained.
\end{itemize}
Note that the expression being minimized does not involve neither the fractional Brownian motion nor the Wiener process,
so the problem becomes purely analytic. %At last, we shall solve the problem $(A'')$: to find
%$$\inf_{a\in  L_2([0,1])}\sup_{t\in [0,1]}\int_0^t(K(t,s)-a(s))^2ds $$ and the minimizing function.

 The  paper is organized as follows. Sections 2 and 3 are devoted to the general problem of minimization of the  functional $f$ on $L_2([0,1])$ that has the following   form
\begin{equation}
f(x) = \sup_{t\in [0,1]} \Big({
\int_0^t (K(t,s)-x(s))^2 \, d s}\Big)^{1/2}
\label{eq-fx}
\end{equation}
with arbitrary  kernel $K(t,s)$ satisfying condition
\begin{itemize}
\item[(B)] for any  $t\in[0,1]$ the kernel $K(t,\cdot)\in L_2([0,t])$ and
\begin{equation}\label{eq_l2}
\sup_{t\in[0,1]}\int_0^t K(t,s)^2ds < \infty.
\end{equation}
\end{itemize}
We shall call this functional the principal functional. It is proved in Section 2 that the principal functional $f$ is convex, continuous and unbounded on infinity,  consequently, the minimum is attained. Section 3 gives an example of kernel $K(t,s)$
where  a minimizing function for principal functional is not unique (moreover, being convex, the set of minimizing functions is infinite). Sections 4--6 are devoted to the problem of minimization of principal functional $f$ with the kernel $K$ corresponding to fractional Brownian motion, i.e., with the kernel $K$ from \eqref{oj4}. It is proved in Section 4 that in  this case  the minimizing function for the principal functional is unique. In Section 5 it is proved that the minimizing function has a special
form. Section 6 contains some numerical results.

\section{The existence of minimizing function for the principal functional}

In this section we consider arbitrary kernel $K$ satisfying assumption (B), which implies that the functional $f$ is well defined for any  $x\in L_2([0,1])$.

\begin{lemma}
For any $x,y\in L_2([0,1])$
\begin{equation}\label{eq_lip}
|f(x) - f(y)| \le \|x-y\| _{L_2([0,1])}.
\end{equation}
\end{lemma}
\begin{proof} Evidently, for any $x, y\in L_2([0,1])$ and  $0\le t\le 1$
%Triangle inequality  implies that for any $t\in [0,1]$
\[
\Big({\int_0^t (K(t,s)-x(s))^2 \, d s}\Big)^{1/2} \le
\Big({\int_0^t (x(s) - y(s))^2 \, d s}\Big)^{1/2} +
\Big({\int_0^t (K(t,s)-y(s))^2 \, d s}\Big)^{1/2}.
\]

Therefore
\begin{gather*}
\sup_{t\in [0,1]} \Big({\int_0^t (K(t,s)-x(s))^2 \, d s}\Big)^{1/2}
\\
 \le
\sup_{t\in [0,1]} \Big({\int_0^t (x(s) - y(s))^2 \, d s}\Big)^{1/2} +
\sup_{t\in [0,1]} \Big({\int_0^t (K(t,s)-y(s))^2 \, d s}\Big)^{1/2},
\end{gather*}
which is clearly equivalent to the inequality
\begin{gather*}
%f(x) \le \Big({\int_0^1 (x(s) - y(s))^2 \, d s}\Big)^{1/2} + f(y),
%\nonumber \\
f(x) \le \|x-y\|_{L_2([0,1])} + f(y).
%\nonumber \\
%f(x)-f(y)\le \|x-y\|_{L_2([0,1])}.
%\nonumber
\end{gather*}
Swapping $x$ and $y$, we get the proof.
%Помінявши місцями $x$ та $y$, отримаємо
%\[
%f(y)-f(x)\le \|x-y\|_{L_2([0,1])}.
%\]
%Останні дві нерівності доводять лему.
\end{proof}
\begin{corollary}\label{cor:f-cont}
The functional $f$ is continuous on $L_2([0,1])$.
\end{corollary}

\begin{lemma}\label{lem:f-ineq1}
The following inequalities hold for any function $x\in L_2([0,1])$:
\begin{equation}\label{ineq1}
\|x\|_{L_2([0,1])} - \|K(1,\cdot)\|_{L_2([0,1])} \le f(x)\le \|x\|_{L_2([0,1])} + f(0).
\end{equation}

\end{lemma}
\begin{proof}
The left-hand side of \eqref{ineq1} immediately follows from the inequalities
\begin{gather*}
f(x) \ge \Big(\int_0^1 (K(1,s)-x(s))^2 \, d s\Big)^{1/2} =
\|K(1,\cdot)-x \|_{L_2([0,1])}
% \nonumber
\\
\ge \|x\|_{L_2([0,1])} - \|K(1,\cdot)\|_{L_2([0,1])},
%\label{neq-luf(x)}
\end{gather*}
and the right-hand side follows from \eqref{eq_lip}.
\end{proof}

\begin{lemma}\label{lem:f-convex}
Functional  $f$ is convex on $L_2([0,1])$.
\end{lemma}
\begin{proof}
We have to prove that for any $x,y\in L_2([0,1])$ and  any $\alpha\in[0,1]$
\[
f(\alpha x + (1-\alpha) y) \le \alpha f(x) + (1-\alpha) f(y) .
\] applying the triangle inequality, we have for any  $t\in [0,1]$
\begin{multline*}
\Big(\int_{0}^t (\alpha x(s) + (1-\alpha) y(s) - K(t,s))^2 \, d s\Big)^{1/2}
\\ \le
\Big(\int_{0}^t (\alpha \,(K(t,s)-x(s)))^2 \, d s\Big)^{1/2} +
\Big(\int_{0}^t ((1-\alpha) (K(t,s)-y(s)))^2 \, d s\Big)^{1/2},
\end{multline*}
whence
\begin{multline*}
\sup_{t \in [0,1]}
\Big(\int_{0}^t (\alpha x(s) + (1-\alpha) y(s) - K(t,s))^2 \, d s\Big)^{1/2}
 \\ \le
\alpha\sup_{t \in [0,1]}  \Big(\int_{0}^t (K(t,s)-x(s))^2 \, d s\Big)^{1/2} +
 (1 - \alpha)\sup_{t \in [0,1]}
\Big(\int_{0}^t (K(t,s)-y(s))^2 \, d s\Big)^{1/2},
\end{multline*}
and the proof follows.%отже
%\begin{equation*}
%f(\alpha x + (1-\alpha) y) \le \alpha f(x) + (1-\alpha) f(y) .
%\qedhere
%\end{equation*}
\end{proof}

%\begin{lemma}
%Функціонал $f$ є опуклим донизу, тобто для довільних $x,y\in L_2([0,1])$, довільного $\alpha\in(0,1)$ має місце нерівність
%\[
%f(\alpha x + (1-\alpha) y) < \alpha f(x) + (1-\alpha) f(y) .
%\]
%\end{lemma}
%\begin{proof}
%\end{proof}
%\iffalse
%Now we established an auxiliary result concerning convex functionals.
%\begin{lemma}\label{lem2.4}(On the weak semi-continuity of the convex functional)
%Let $E$ be the normed linear space,\ssf(2000)
%$f: E\to \bbR$ be the continuous convex functional and
%the sequence $\{x_n, n\ge 1\}$
%weakly converges in $E$ to $x_\infty$.
%Then
%\[ \varliminf_{n\to\infty} f(x_n) \ge f(x_\infty) .
%\]
%\end{lemma}
%\begin{proof}
%It follows from the convexity and continuity of the functional $f$ that at  point $x_\infty$
%there exists the supporting plane to its graph. In other words, there exists
%linear and continuous functional $T \in E^*$ such  that
%\begin{equation}\label{equ1}
%f(x) \ge f(x_\infty) + T(x - x_\infty) =
%f(x_\infty) + T(x) - T(x_\infty)\quad
%\text{for any $x\in E$}
%\end{equation}
%(See e.g.   \cite{Young}, Theorem 42.3, p. 99.)
%Substitute  $x = x_n$ into \eqref{equ1} and obtain
%
%\begin{equation}\label{equ2}
%f(x_n) \ge
%f(x_\infty) + T(x_n) - T(x_\infty).
%\end{equation}
%Taking into account that \[
%T(x_n) \to T(x_\infty), \quad n\to\infty,
%\]
%we obtain from \eqref{equ2} that
%\[
%\varliminf_{n\to\infty} f(x_n) \ge f(x_\infty)+ \varliminf_{n\to\infty} T(x_n)-T(x_\infty)=f(x_\infty),
%\]
%whence the proof follows.
%\end{proof}
%\fi

\begin{theorem}\label{thm_min_ex}
Functional  $f$  attains its minimal value on $L_2([0,1])$.
\end{theorem}

\begin{proof}
By Corollary~\ref{cor:f-cont} and
Lemma  \ref{lem:f-convex}
the functional $f$ is continuous and convex.  By Lemma \ref{lem:f-ineq1},  $f(x)$ tends to $+\infty$
as $\|x\| \to \infty$.
Hence it follows from \cite[Proposition~2.3]{Bashirov} that $f$ attains its minimal value.
\end{proof}
%   \begin{proof}
%   Evidently,
%   \[
%   0\leq \inf_{x \in L_2([0,1])} f(x)<\infty.
%   \]
%   There exists a sequence
%   $\{x_n,\allowbreak\,\; n{\ge}1\}$ from  $L_2([0,1])$
%   such that
%   \begin{equation}
%   f(x_n) \to \inf_{x \in L_2([0,1])} f(x).
%   \label{L0163}
%   \end{equation}
%   Sequence $\{f(x_n),\ n{\ge}1\}$ is bounded, and it follows
%   from  (\ref{neq-luf(x)}) that the sequence $\{x_n, \ n{\ge}1\}$
%   is bounded as well:
%   \[
%   \|x_n\|_{L_2([0,1])} \le R, \qquad n{=}1,2,\ldots
%   \]
%   where $\displaystyle R = \sup_{n\ge 1} f(x_n) + \|K(1,\cdot)\|_{L_2([0,1])}$.
%
%   Unit ball in  $L_2([0,1])$ is weakly compact,
%   the ball of radius  $R$ with the center in the origin is weakly compact as well,
%   therefore the sequence
%   $\{x_n, n{\ge} 1\}$
%   contains the weakly convergent subsequence
%   $\{x_{n_k}, n{\ge} 1\}$:
%   $
%   x_{n_k} \weakto x_\infty,
%   $
%   where $x_\infty \in L_2([0,1]).$%,\ssf(3000) $\|x_\infty\|_{L_2([0,1])}<\infty$.
%
%   It follows from Lemma \ref{lem2.4} that
%   \[
%     \varliminf_{k\to\infty} f(x_{n_k}) \ge f(x_\infty).
%   \]
%   Taking  (\ref{L0163}) into account, we get that  $\inf_{x\in L_2([0,1])} f(x) \ge f(x_\infty)$.
%   It means that functional $f$ achieves its minimal value at the point
%   $x_\infty$.
%   \end{proof}

\section{An example of the principal functional with infinite  set of minimizing functions}
Note that the set $\mathfrak{M}_f$ of minimizing functions for functional $f$ is convex. In this section we consider an example of kernel $K$ for which $\mathfrak{M}_f$ contains more than one point, consequently, is infinite. At first, establish the following lower bound for functional $f$.
\begin{lemma}\label{lem_K_ineq}
1.  Let the kernel  $K$ of functional  $f$ defined by \eqref{eq-fx} satisfy assumption $(B)$. Then for any  $a \in L_2([0,1])$ and  $0 \le t_1<t_2 \le 1$ the following inequality holds
\begin{equation}
\label{eq-sj1}
\sup_{t\in[0,1]} \int_0^t (K(t,s) - a(s))^2 ds \ge
\frac14 \int_0^{t_1} (K(t_2,s) - K(t_1,s))^2 ds .
\end{equation}

2. The equality in \eqref{eq-sj1} implies that
\begin{gather}
\label{eq-sj2}
a(s) = 1/2(K(t_1,s) + K(t_2,s)) \quad \text{a.e. on  $[0,t_1)$},\\
\label{eq-sj3}
a(s) = K(t_2,s) \quad \text{a.e. on $[t_1,t_2]$}.
\end{gather}
\end{lemma}

\begin{proof}
1. Following inequalities are evident:
\begin{multline}\label{oj1}
\sup_{t\in[0,1]} \int_0^t (K(t,s)-a(s))^2 ds  \\ \ge
\max\left\{
\int_0^{t_1} (K(t_1,s)-a(s))^2 ds, \;
\int_0^{t_2} (K(t_2,s)-a(s))^2 ds \right\} \\ \ge
\max\left\{
\int_0^{t_1} (K(t_1,s)-a(s))^2 ds, \;
\int_0^{t_1} (K(t_2,s)-a(s))^2 ds \right\}\\ \ge %\\ \ge
%\frac{\int_0^{t_1} (K(t_1,s)-a(s))^2 ds +
%\int_0^{t_1} (K(t_2,s)-a(s))^2 ds}{2} = \\ =
\frac12\int_0^{t_1}
({(K(t_1,s)-a(s))^2 + (K(t_2,s)-a(s))^2}) \, ds .
\end{multline}

From $(P+Q-2r)^2\ge 0$ we immediately get
\begin{gather}
2\left(\frac{P-r}{2}\right)^2 +
2\left(\frac{Q-r}{2}\right)^2% =
%\left(\frac{A-a}{2}-\frac{B-a}{2}\right)^2 +
%\left(\frac{A-a}{2}+\frac{B-a}{2}\right)^2 ,
%\nonumber \\
%\frac{(A-a)^2 + (B-a)^2}{2} =
%% \frac{2A^2 -4Aa + 2a^2 + 2B^2 - 4aB + 2a^2}{4} = \\ =
%% \frac{A^2 - 2AB + B^2 + A^2 + 2AB + B^2 - 4aA - 4aB + 4a^2}{4} = \\ =
%\frac{(A-B)^2}{4} + \frac{(A+B-2a)^2}{4}
 \ge
\frac{(P-Q)^2}{4}.
\label{eq-sj4}
\end{gather}
Setting
$P = K(t_1,s)$, $Q=K(t_2,s)$ and  $r = a(s)$ in this inequality, we get from \eqref{oj1}
\begin{multline}
\sup_{t \in [0,1]} \int_0^t (K(t,s) - a(s))^2 ds \ge
\frac12\int_0^{t_1} ({(K(t_1,s)-a(s))^2 + (K(t_2,s)-a(s))^2}) \, ds  \\ \ge
%\int_0^{t_1} \frac{(K(t_1,s)-K(t_2,s))^2}{4} \, ds =
\frac14 \int_0^{t_1} (K(t_2,s) - K(t_1,s))^2 ds .
\label{eq-sj5}
\end{multline}
Thus, inequality  \eqref{eq-sj1} is proved.

2. We now show that equality in \eqref{eq-sj1} implies
 \eqref{eq-sj2} and  \eqref{eq-sj3}.
Indeed, equality in \eqref{eq-sj4}  holds if and only if
$P+Q-2r=0$.
Equality in  \eqref{eq-sj5} has a form
\[
1/2\int_0^{t_1} (K(t_1,s)-a(s))^2 + (K(t_2,s)-a(s))^2 \, ds =
1/4\int_0^{t_1} (K(t_1,s)-K(t_2,s))^2 \, ds
\]
and holds if and only if
\[
K(t_1,s) + K(t_2,s) - 2 a(s) = 0 \quad
\mbox{a.e. on $[0,t_1)$,}
\]
i.e.  it holds if and only if   condition  \eqref{eq-sj2} holds.

If \eqref{eq-sj2} holds, then
\begin{gather*}
\int_0^{t_1} (K(t_1,s) - a(s))^2 ds =
\frac{1}{4} \int_0^{t_1} (K(t_1,s) - K(t_2,s))^2 ds,
\end{gather*} and
\begin{gather*}
\int_0^{t_2} (K(t_2,s) - a(s))^2 ds =
\frac{1}{4} \int_0^{t_1} (K(t_2,s) - K(t_1,s))^2 ds +
\int_{t_1}^{t_2} (K(t_2,s) - a(s))^2 ds.
\end{gather*}
It means that under condition  \eqref{eq-sj2} equality  \eqref{eq-sj1} holds only if
\begin{gather*}
\int_{t_1}^{t_2} (K(t_2,s) - a(s))^2 ds = 0,
\end{gather*}
i.e. only if \eqref{eq-sj3} holds.
\end{proof} %\\
%a(s) = K(t_2,s) \quad
%\mbox{a.e. on $[t_1,t_2]$,}

\begin{remark} Let the kernel  $K$ of functional  $f$ from \eqref{eq-fx} satisfy assumption (A). Then for any  $a \in L_2([0,1])$ and  $0 \le t_1<t_2 \le 1$
\begin{equation}
\label{oj2}\max_{t\in\{t_1,t_2\}}
\int_0^t (K(t,s)-a(s))^2 ds \ge
\frac14 \int_0^{t_1} (K(t_2,s) - K(t_1,s))^2 ds.
\end{equation}
Equality in \eqref{oj2} holds if and only if (\ref{eq-sj2}) and  (\ref{eq-sj3}) hold.
\end{remark}

\begin{example}[Functional $f$ with infinite set $\mathfrak{M}_f$.]
Take the kernel  $K(t,s)$ of the form
$K(t,s) = g(t) h(s)$, $t, s \in [0,1]$, where
$$g(t) = (6t-2)1_{\frac13 \le t \le \frac12}+(4-6t)1_{\frac12\le t \le \frac56}+(6t-6)1_{\frac56 \le t \le 1}$$
and $$h(s) = 4s1_{0\le s \le \frac14}+(2-4s)1_{\frac14 \le s \le \frac12}.$$

Then
\begin{equation}
\min_{a\in L_2([0,1])}
\max_{t\in [0,1]}
\int_0^t (K(t,s)-a(s))^2 ds = 1/6,
\label{eq-sj-minmax16}
\end{equation}
and  $\mathfrak{M}_f$ consists of functions $a(s)$
satisfying the conditions
\begin{gather}\label{eq-sjstar1}
a(s) = 0 \quad \text{a.e. on  $[0,5/6]$}\end{gather}
 and
\begin{gather}\int_{5/6}^t a(s)^2 ds \le 1/6  - 6(1-t)^2, \quad
5/6 \le t \le 1 .
\label{eq-sjstar2}
\end{gather}
\begin{remark} 1. Since  $K\in C([0,1]^2)$ and  $a\in L_2([0,1])$, we have that
$\int_0^t (K(t,s) - a(s))^2 ds$ is continuous in
$t$, therefore we can replace $\sup_{t\in[0,1]}$ with $\max_{t\in[0,1]}$ in inequality (\ref{eq-sj-minmax16}).

2. Some examples of functions satisfying  (\ref{eq-sjstar1}) and  (\ref{eq-sjstar2}):
$a(s) = 0,  s\in [0,1];$
$a(s) =
  (12(1-s))^{1/2}1_{5/6 < s \le 1};$
$a(s)=
  \sqrt{3} (6s-5)1_{5/6\le s \le 1}.$
\end{remark}

To establish a lower bound on the left-hand side of  \eqref{eq-sj-minmax16}, note that
$
\int_0^t h(s)^2 ds = 1/6
\quad \mbox{for} \quad 1/2 \le t \le 1.
$
Therefore, applying  Lemma \ref{lem_K_ineq} with $t_1=1/2$ and $t_2=5/6$ we obtain that
\begin{multline}
\sup_{t \in [0,1]} \int_0^t (K(t,s)-a(s))^2 ds \ge
\frac14 \int_0^{1/2}
\left( K({\textstyle5/6},s) - K({\textstyle1/2},s)\right)^2 ds
 \\ =
\frac14\int_0^{1/2}
\left( g({\textstyle5/6})h(s) - g({\textstyle1/2})h(s)\right)^2 ds =
\frac14\int_0^{1/2} 4 h(s)^2 ds = 1/6.
\label{eq-sj-uselemma}
\end{multline}
Moreover, functions   $a(s)$ satisfying
$(\ref{eq-sjstar1})$ and  $(\ref{eq-sjstar2})$ transform \eqref{eq-sj-uselemma} into equality.

To establish an upper bound of the left-hand side of  \eqref{eq-sj-minmax16}, consider functions satisfying conditions (\ref{eq-sjstar1}) and  (\ref{eq-sjstar2}). Then for  $0\le t \le 5/6$ we have that
\begin{multline*}
\int_0^t (K(t,s) - a(s))^2 ds =
\int_0^t K(t,s)^2 ds =
\int_0^t g(t)^2 h(s)^2 ds  \\ =
g(t)^2 \int_0^t h(s)^2 ds \le
\int_0^{5/6} h(s)^2 ds = 1/6,
\end{multline*}
since  $a(s)=0$ on $[0,5/6]$ and  $g(t)^2 \leq 1$.
For  $5/6 < t \le 1$, we take into account the values  of $a,h$ and $g$ on this interval and obtain that
\begin{multline}
\int_0^t (K(t,s) - a(s))^2 ds %
=
%\int_0^{5/6} (K(t,s) - a(s))^2 ds +
%\int_{5/6}^t (K(t,s) - a(s))^2 ds  \\ =
\int_0^{5/6} (g(t) h(s) - a(s))^2 ds +
\int_{5/6}^t (g(t) h(s) - a(s))^2 ds  \\ =
\int_0^{5/6} g(t)^2 h(s)^2 ds +
\int_{5/6}^t a(s)^2 ds =
g(t)^2 \int_0^{5/6} h(s)^2 ds +
\int_{5/6}^t a(s)^2 ds  \\ \le
(6t-6)^2 \cdot 1/6 + 1/6 - 6(1-t)^2 = 1/6.
\label{eq-sj6}
\end{multline}
%бо $a(s)=0$ майже скрізь на $[0,5/6]$,
%$h(s)=0$ при $s\ge5/6$ (навіть при $s\ge1/2$),
%$g(t)=(6t-6)$ при $t\ge5/6$.
Hence, if function $a$ satisfies
(\ref{eq-sjstar1}) and  (\ref{eq-sjstar2}), we have that
\begin{gather*}
%\int_0^t (K(t,s) - a(s))^2 ds \le 1/6
%\quad \mbox{при} \quad 0\le t \le 1, \\
\sup_{t\in[0,1]} \int_0^t (K(t,s) - a(s))^2 ds \le 1/6.
\end{gather*}
Summing up, we obtain \eqref{eq-sj-minmax16}.
%\subparagraph*{\it

Now we prove that any minimizing function $a$ satisfies $(\ref{eq-sjstar1})$ and  $(\ref{eq-sjstar2})$.%}\mbox{}

Indeed, let
\[
\sup_{t\in[0,1]} \int_0^t (K(t,s) - a(s))^2 ds = 1/6.
\]
Then inequality (\ref{eq-sj-uselemma}) is transformed into equality, therefore
\begin{equation}\label{oj3}
\sup_{t\in[0,1]} \int_0^t (K(t,s) - a(s))^2 ds =
\frac14 \int_0^{1/2}
\left( K({\textstyle5/6},s) - K({\textstyle1/2},s)\right)^2 ds .
\end{equation}
It follows from \eqref{oj3} and from the 2nd part of Lemma~\ref{lem_K_ineq}  that
\[
a(s) = \frac12 (K({\textstyle5/6},s) + K({\textstyle1/2},s))
= \frac12(g({\textstyle5/6}) + g({\textstyle1/2}))\,
h(s) = 0
\]
a.e. on $[0, 1/2]$ because $g(1/2)=1$, $g(5/6)=-1$;
we obtain also the equality
\[
a(s) = K({\textstyle5/6},s) = g({\textstyle5/6}) h(s) = 0
\]
a.e. on $[1/2, 5/6]$ because  $h(s)=0$ for  $s\ge1/2$.
Therefore, function $a$ satisfies condition  (\ref{eq-sjstar1}).
Then we can get similarly to (\ref{eq-sj6}) that
\[
\int_0^t (K(t,s)-a(s))^2 ds =
\frac{(6t - 6)^2}{6} + \int_{5/6}^t a(s)^2 ds
\quad \mbox{for} \quad 5/6 < t \le 1,
\]
and it follows from inequality $\int_0^t (K(t,s) - a(s))^2 ds \le 1/6$
that
\[
\int_{5/6}^t a(s)^2 ds \le
1/6 - \frac{(6t - 6)^2}{6} = 1/6 - 6(1-t)^2
\quad \mbox{for} \quad 5/6 < t \le 1.
\]
It means that function $a$ satisfies condition (\ref{eq-sjstar2}).
\end{example}

\section{Uniqueness  of the minimizing function for the kernel connected to fractional Brownian motion}

Now we return to the main problem $(A)$ of approximation of fractional Brownian motion by martingales.

First we prove some simple but useful properties of the fractional Brownian kernel $K$ defined by \eqref{oj4}.
\begin{lemma}[Properties of the fractional Brownian kernel]\label{lem_K_prop}
1. Kernel  $K$ satisfies condition $(B)$.

2. Kernel  $K$ increases in the first argument and decreases in the second argument.

3. Kernel  $K$ is continuous on the set $[0,1]\times(0,1]$.

4. For any  $c>0$ and  $0<s\le t$ we have that $K(ct,cs) = c^\alpha K(t,s)$ with $\alpha=H-1/2.$
\end{lemma}
\begin{proof}
1. Since $K$ is the kernel of fractional Brownian motion, we have that
\[
t^{2H}=\M(B^H_t)^2 = \M\left(\int_0^t K(t,s) dW_s\right)^2 = \int_0^t K(t,s)^2ds.
\]
Therefore,  $\sup_{t\in [0,1]} \int_0^t K(t,s)^2 ds = 1$, and \eqref{eq_l2}. Other statements  follow directly from \eqref{oj4}.
\end{proof}

%The following  simple result  implies that the problem (A) is equivalent to  the following minimization problem (A'): to find
%\begin{equation}\label{oj5}\inf_{x\in  L_2([0,1])}\sup_{t\in [0,1]}\int_0^t(K(t,s)-x(s))^2ds \end{equation}
%and the minimizing function $a$. Note that problem $(A'')$ is the problem of minimization of principal functional $f(x) = \sup_{t\in [0,1]} \Big({
%\int_0^t (K(t,s)-x(s))^2 \, d s}\Big)^{1/2}$ with the kernel $K$ from \eqref{oj4}. Since the kernel $K$ satisfies condition (B), we can apply Theorem \ref{thm_min_ex} and obtain that infimum in \eqref{oj5} is attained on some set $\mathcal{M}_f\in L_2([0,1])$.
\begin{theorem}
For any function $a \in \mathfrak{M}_f$ there exists such function $\phi:[0,1]\rightarrow\R$  that  $s\le \phi(s)\le 1, s\in[0,1]$ and  $a(s)=K(\phi(s),s)$ a.e.
\end{theorem}
\begin{proof}
Let $a \in \mathcal{M}_f$.
Consider the function $b(s)=\min(K(1,s), \max(0, a(s)), s\in[0,1]$.
Since the kernel  $K$ is nonnegative, then  $$(a(s)-K(t,s))^2 \ge (\max(0, a(s))-K(t,s))^2, t,s\in [0,1] $$
and this inequality is strict on a set of positive Lebesgue measure if $a(s)<0$ on a set of positive Lebesgue measure.  Moreover, since the kernel  $K$ is increasing in the first argument, we have that
$$(a(s)-K(t,s))^2 \ge (\min(K(1,s), a(s))-K(t,s))^2, t,s\in [0,1],$$ and this inequality is strict on the set of positive Lebesgue measure if $a(s)>K(1,s)$ on a set of positive Lebesgue measure.
Therefore,  $f(b)\le f(a)$ and this inequality is strict if $a(s)<0$ or $a(s)>K(1,s)$ on a set of positive Lebesgue measure. Therefore,  $$0=K(s,s)\le a(s)\le K(1, s), s\in[0,1].$$ Since the kernel  $K$ is continuous in the first argument, there exists a function  $s\le\phi(s)\le 1,s\in[0,1]$, such that  $a(s)=K(\phi(s),s)$.
\end{proof}
\begin{corollary} Functions in the set $\mathfrak{M}_f$ are nonnegative.
\end{corollary}

Now we are in position to establish the uniqueness of minimizing function for the principal functional corresponding to the kernel of fractional Brownian motion. In order to do this, prove at first the auxiliary statement concerning any minimizing function for this functional. For $x\in L_2([0,1])$, denote
$$g_x(t)=\left(
\int_0^t (K(t,s)-x(s))^2 \, d s\right)^{1/2}.$$
Then we have from the definition of the principal functional $f$ that  $f(x)=\sup_{t\in[0,1]}g_x(t).$
It follows from Lemma \ref{lem_K_prop} that
 $g_x\in C[0,T]$ for any $x\in L_2[0,T]$.
Using self-similarity property  4) of the kernel $K$, it is easy to see that
\begin{equation}\label{selfsim}
g_a(t) =  c^{\alpha+1/2}g_{c^{-\alpha} a(c\cdot)}(t/c).
\end{equation}
%Надалі у цьому розділі $a$ --- це функція, функціонал \eqref{eq-fx} з дробово-броунівським ядром набуває мінімального значення.

\begin{lemma}\label{lem_sup_point} Let $a \in \mathfrak{M}_f$.
Then the maximal value of  $g_a$ is attained at the point $1$, i.e. $f(a)=g_a(1)$.
\end{lemma}
\begin{proof}
Set $a(t)=0$ for $t>1$. Suppose that $g_a(1) < f(a)$. Since $g_a(t)$ is continuous in $t$, there exists such $c>1$ that $g_{a}(t)<f_{a}$ for $t\in[1,c]$. It means that $\max_{t\in[0,c]} g_{a}(t) = f_{a}$. Set $b(t) = c^{-\alpha}a(tc)$. It follows from equation  \eqref{selfsim} that $g_b(t) = c^{-1/2-\alpha} g_{a}(tc)$, $t\in[0,1]$. We get immediately that  $f(b)= c^{-\alpha-1/2} f(a)< f(a)$, which leads to a contradiction.
\end{proof}

%Визначимо $a_1(s) = K(1,s)$,
%$$
%g_1(s) = g_{a_1}(s) = \sqrt{\int_0^t (K(1,s) - K(t,s))^2 \, d s}.
%$$
%Ця функція є неперервною на $[0,1]$, причому $g_1(0)=g_1(1)=0$.
%
%\begin{lemma}
%Якщо $a\in L_2([0,1])$ --- функція, на якій функціонал $f(a)$ набує мінімального значення, тоді знайдеться така точка $t\in(0,1)$, що $g_1(t)\ge f(a)$, $g_a(t) = f(a)$.
%\end{lemma}
%\begin{proof}
%Нехай $A = \{t\in(0,1): g_1(t)\ge f(a)\}$. Оскільки $g_1\in C[0,1]$, $g_1(0)=g_1(1)=0$ та $f(a)>0$, то множина $A$ є замкненою.
%
%Припустимо, що твердження леми не виконано, тоді $m := \max_{s\in M} g_a(s)<f(a)$. Позначимо $M=\max_{s\in M} g_1(s)$. Тоді для певного $\gamma\in(0,1)$ $\gamma M +(1-\gamma)m <f(a)$. Визначимо $b_\gamma(s) = \gamma K(1,s) + (1-\gamma) a(s)$.
%Запишемо
%\begin{gather*}
%g_{b_\gamma}(t) = \sqrt{
%\int_0^t (\gamma K(1,s) + (1-\gamma) a(s) - K(t,s))^2 \, d s}=\\= \sqrt{
%\int_0^t \left(\gamma \big(K(1,s)-K(t,s)\big) + (1-\gamma)\big( a(s) - K(t,s)\big)\right)^2 \, d s}\le \\
%\le (1-\gamma)g_1(t) + \gamma g_a(t).
%\end{gather*}
%При $t\in [0,1]\setminus A$ маємо $g_1(t)<f(a)$, тому
%$g_{b_\gamma}(t) < f(a)$, при $t\in A$ \ $g_{b_\gamma}(t)\le \gamma M + (1-\gamma) m<f(a)$. Звідси $f(b_\gamma)<f(a)$, що суперечить припущенню
%про мінімальність $f(a)$.
%\end{proof}

\begin{theorem}[Uniqueness of minimizing function]
For the  principal functional $f$ defined by \eqref{eq-fx} with fractional Brownian kernel $K$ from \eqref{oj4},
 there is a unique minimizing function.
\end{theorem}
\begin{proof}
Denote $M_f$ the minimal value of functional $f$. Recall that the set $\mathfrak{M}_f$
is nonempty and convex.
%, та для довільних $x,y\in A, \alpha \in [0,1]$ виконується рівність$$f(\alpha x + (1-\alpha)y)=$$
Let $\hat{K}(s)=K(1,s),s\in[0,1].$
It follows from Lemma \ref{lem_sup_point} that for any function  $x\in \mathfrak{M}_f$ the following equality holds:
$$f(x)=\left(\int_{0}^1 (x(s) - K(1,s))^2 ds\right)^{1/2}=\|x-\hat{K}\|_{L_2([0,1])}.$$

For any $x,y\in \mathfrak{M}_f, \alpha\in (0,1)$ we have that
\begin{multline*}
M_f=f(\alpha x + (1-\alpha)y)=\|\alpha x + (1-\alpha)y-L\|_{L_2([0,1])}\le \alpha \| x - L \|_{L_2([0,1])}\\ + (1-\alpha)\|y-L\|_{L_2([0,1])}
=\alpha f(x) + (1-\alpha) f(y) = M_f.
\end{multline*}
For arbitrary vectors $x$ and $y$ in a Hilbert space the equality $\|x+y\|=\|x\|+\|y\|$ implies that  $x$ and $y$ differ by
a non-negative multiple.
Therefore, the functions  $\hat{K}-x$ and $\hat{K}-y$ differ by a non-negative multiple, but since $\|\hat{K}-x\|_{L_2([0,1])}=\|\hat{K}-y\|_{L_2([0,1])}$, we have $\hat{K}-x=\hat{K}-y$. Therefore, $x=y$, as required.
\end{proof}

\section{Representation of the minimizing function}

In this section we consider principal functional $f$ corresponding to fractional Brownian motion and establish
that the minimizing function has some special form. We start by proving several auxiliary results of the fractional Brownian kernel and the minimizing function.

\subsection{Auxiliary results}

\begin{lemma}\label{lem-LF}
The fractional Brownian kernel for any $0 \le t \le 1$ satisfies
\begin{equation}\label{eq-lem-LF}
\int_0^t (K(1,s)-K(t,s))^2 ds + \int_t^1 K(1,s)^2 ds = (1-t)^{2H} .
\end{equation}
\end{lemma}

\begin{proof}
It follows from \eqref{fBm-repres} that the left-hand side of \eqref{eq-lem-LF} is equal to $\M (B_1^H-B_t^H)^2=(1-t)^{2H}$.
\end{proof}

The following statement will be essentially generalized in what follows. However, we prove it because its proof clarifies the main ideas and, moreover, it has the interesting consequences concerning the properties of the minimizing function. In the remainder of this section $a=a(s), s \in [0,1]$ denotes the minimizing function, i.e. the unique element of $\mathfrak{M}_f$.

\begin{lemma}\label{cikavetverdz}
Let $t^* = \sup\{t\in(0,1): g_a(t) = f(a)\}$ $(t^*=0$ if this set is empty$)$.
If $t^*<1$, then $a(t) = K(1,t)$ for a.e. $t\in [t^*,1]$.
\end{lemma}
\begin{proof}
Fix some $t_1\in(t^*,1]$ and prove that for any  $h\in L_2([0,1])$ the following equality  holds: $$\int_{t_1}^1 h(s) \left(a(s)-K(1,s)\right) ds =0.$$ Evidently, proof follows immediately from this statement.

Assume the contrary. Then, without loss of generality, there exists such $h\in L_2([0,1])$ that $$\int_{t_1}^1 h(s)\big(a(s)-K(1,s)\big) ds =:\kappa>0.$$ It follows from the continuity of the last integral w.r.t. upper bound that for some $t_2\in(t_1,1]$ we have  $$\int_{t_1}^t h(s) \left(a(s)-K(t,s)\right) ds \ge \kappa/2$$ for any
 $t\in[t_2,1]$. Note also that our assumption implies that  $$m:=\max_{s\in[t_1,t_2]} g_a(s)< f(a).$$

Consider now   $b_\delta (t) = a(t) - \delta h(t)1_{[t_1,1]}(t)$ for $\delta>0$.
We have that  $g_{b_\delta}(t) = g_a(t)$ for $t\in[0,t_1]$, and
\begin{gather*}
g_{b_\delta}(t)^2 = g_a(t)^2 - 2\delta\int_{t_1}^t h(s)(a(s)-K(t,s))\,ds + \delta^2 \int_{t_1}^t h(s)^2 ds
\end{gather*}
for $t > t_1$.
For $t\in(t_1,t_2]$ the following inequality holds,
\begin{gather*}
g_{b_\delta}(t)^2 \le m^2 - 2\delta \int_{t_1}^t h(s)\big(a(s)-K(t,s)\big) ds + \delta^2 \int_{t_1}^t h(s)^2 ds\le m^2 + C\delta
\end{gather*}
with the constant $C$ that does not depend on $t,\delta$. Then for   sufficiently small $\delta>0$ we have that
$g_{b_\delta}(t)< f(a)$
for any $t\in(t_1,t_2]$.

Furthermore, if $t\in(t_2,1]$, then
\begin{gather*}
g_{b_\delta}(t)^2 \le f(a)^2 - 2\delta \int_{t_1}^t h(s)\big(a(s)-K(t,s)\big) ds + \delta^2 \int_{t_1}^t h(s)^2 ds\le\\
\le f(a)^2 -\kappa \delta + \delta^2 \int_0^1 h(s)^2 ds.
\end{gather*}
Again, for   sufficiently small $\delta>0$ and any  $t\in(t_2,1]$ we have that
$g_{b_\delta}(t)< f(a)$.
Therefore, for   sufficiently small  $\delta>0$ we get that $f(b_\delta)= f(a)$ and  $g_{b_\delta}(1)<f(a) = f(b_\delta)$. We obtain the contradiction with  Lemma ~\ref{lem_sup_point} whence the proof follows.
\end{proof}
\begin{corollary}
 There exists such point $t\in(0,1)$ that
$g_a(t)=f(a)$.
\end{corollary}
\begin{proof}
Assuming the  contrary, we get from Lemma~\ref{cikavetverdz} that $a(t) =K(1,t)$ for a.a. $t\in[0,1]$. However, in this case  $g_a(1)=0$, which contradicts Lemma~\ref{lem_sup_point}.
\end{proof}
Denote $\mathfrak{G}_a =\{t\in[0,1]: g_a(t) =f(a)\}$, the set of the maximal points of the function   $g_a$.
%Міркуваннями, аналогічними до тих, що використані в доведенні твердження~\ref{cikavetverdz}, одержуємо наступну лему.
\begin{lemma}\label{propersep}
Let point $u\in[0,1)$ is such that $g_a(u)<f(a)$.
Then there does not exist function $h\in L_2([0,1])$ such that for any  $t\in \mathfrak{G}_{a}\cap (u,1]$ the inequality $\int_{u}^t h(s)\big(a(s)-K(t,s)\big)ds >0$ holds.
\end{lemma}
\begin{proof}
Assume the contrary, i.e. let for some function $h\in L_2([0,1])$ we have that  $\int_{u}^t h(s)\big(a(s)-K(t,s)\big)ds >0$ for any  $t\in \mathfrak{G}_{a}\cap (u,1]$.   The set $\mathfrak{G}_{a}\cap (u,1]$ is closed because $g_a(u)<f(a)$. Therefore $$\kappa:=\min_{t\in \mathfrak{G}_{a}\cap (u,1]}\int_{u}^t h(s)\big(a(s)-K(t,s)\big)ds >0.$$ Denote $$\mathfrak{B}_\eps = \{t\in (u,1]: \mathfrak{G}_{a}\cap (u,1]\cap (t-\eps,t+\eps)\neq \varnothing\}$$
the intersection of $\eps$-neighborhood of the set $\mathfrak{G}_{a}\cap (u,1]$ with interval $(u,1]$. Continuity argument implies that for some   $\eps>0$ it holds that
$$
\int_{u}^t h(s)\big(a(s)-K(t,s)\big)ds>\kappa/2
$$
for any $t\in \mathfrak{B}_\eps$.
Similarly to the proof of Lemma ~\ref{cikavetverdz}, denote   $b_\delta (t) = a(t) - \delta h(t)1_{(u,1]}(t)$ for any $\delta>0$.
Then we have that  $g_{b_\delta}(t) = g_a(t)$ for any $t\in[0,u]$, and
\begin{gather*}
g_{b_\delta}(t)^2 \le f(a)^2 - 2\delta \int_{u}^t h(s)\big(a(s)-K(t,s)\big) ds + \delta^2 \int_{u}^t h(s)^2 ds\le\\
\le f(a)^2 -\kappa \delta + \delta^2 \int_0^1 h(s)^2 ds
\end{gather*}
for any $t \in \mathfrak{B}_\eps$.
It follows from the continuity of  $g_a$ that $m=\max_{t\in [u,1]\setminus \mathfrak{B}_\eps} g_a(t)<f(a)$. Therefore
we have for  $t \in (u,1]\setminus \mathfrak{B}_\eps$ that
\begin{gather*}
g_{b_\delta}(t)^2 = g_a(t)^2 - 2\delta \left(a(s)-K(1,s)\right) + \delta^2 \int_{t_1}^t h(s)^2 ds\le\\
\le m^2 - 2\delta \int_{t_1}^t h(s)\big(a(s)-K(t,s)\big) ds + \delta^2 \int_{t_1}^t h(s)^2 ds\le m^2 + C\delta,
\end{gather*}
with the constant $C$ that does not depend on $t$ and $\delta$. It follows from the above bounds that for sufficiently small  $\delta>0$ and for any $t\in(u,1]$ we have the inequality
$g_{b_\delta}(t)< f(a)$.
It means that for sufficiently small $\delta>0$ we get the equality $f(b_\delta)= f(a)$, and moreover, $g_{b_\delta}(1)<f(a) = f(b_\delta)$, which contradicts  Lemma~\ref{lem_sup_point}.
\end{proof}

Lemma  \ref{propersep} supplies the form of minimizing function on the part of the interval  $[0,1]$. All equalities below are considered a.s.
\begin{lemma}\label{st-2.exxi}
Let $t_1 = \min\{t\in(0,1): g_a(t) = f(a)\}$. Then there exist  $t_2\in(t_1,1]\cap \mathfrak{G}_{a}$ and random variable  $\xi$ with the values in $[t_1,t_2]\cap \mathfrak{G}_{a}$ such that for
 $t\in [0,t_2)$  we have that  $P(\xi\ge t)>0$, and the equality
$$
a(t) = \M[K(\xi,t)|\xi\ge t]
$$ holds.
\end{lemma}
\begin{proof}
Consider the set of functions
$$
\mathcal K = \left\{ k_t(s) = K(t,s)1_{s\le t} + a(s)1_{s>t}, t\in \mathfrak{G}_{a} \right\}
$$
and let
$$
\mathcal C = \left\{ \int_0^1 k_t(s) F(dt),\, F\text{ is the distribution function on $\mathfrak{G}_{a}$} \right\}
$$
be the closure of the convex hull of $\mathcal K$. According to Lemma~\ref{propersep}, applied to  $u=0$, there does not exist $h\in L_2([0,1])$ such that
$(h,k)<(h,a)$ for any $k\in\mathcal K$. Moreover,  there is no $h\in L_2([0,1])$ such that $(h,k)<(h,a)$ for any  $k\in\mathcal C$, i.e. the element $a$ and the set  $\mathcal K$ can not be separated properly. Then, according to the proper separation theorem
(see e.g. \cite[Corollary 4.1.3]{convan}), $a\in \mathcal C$, so there exists such distribution  $F$ on
$\mathfrak{G}_{a}$ that
\begin{equation}\label{umozh-1}
a(s) = \int_0^1 k_t(s) G(dt) = \int_{[s,1]} k_t(s) F(dt) + \int_{[0,s)} a(s) F(dt).
\end{equation}
Hence
\begin{equation}\label{umozh}
a(s)F([s,1]) = \int_{[s,1]} k_t(s) F(dt).
\end{equation}
Note that the equality $\supp F=\{t_1\}$ is impossible because otherwise it follows from equation \eqref{umozh} that $a(s) = K(t_1,s)$ for $s\le t_1$, therefore $g_a(t_1)=0$ which contradicts the assumption $g_a(t) = f(a)$.

Using the latter statement and  $\eqref{umozh}$, we get the statement of the theorem with $t_2 = \max(\supp F)$ and random variable $\xi$ with the distribution $F$.
\end{proof}

Conditions on minimizing function from Lemma~\ref{st-2.exxi} are sufficient in the following sense.
\begin{lemma}\label{thm-cond_a_min}
Let $y\in\LT$. Define the kernel  $K_y(t,s)$ for $s,t \in[0,1]$ as
\[
K_y(t,s) = \begin{cases}
K(t,s) & \mbox{\rm for} \quad t\ge s, \\
y(s)   & \mbox{\rm for} \quad t<s .
\end{cases}
\]
Function $y$ is the minimizing function of the principal functional  $f $ if and only if there exists random variable
 $\xi$ taking values in  $[0,1]$ such that the following conditions hold:
\begin{gather}
\label{cond-a_min_a}
y(s) = \M K_y(\xi,s) \qquad \mbox{a.a. $s\in[0,1]$},\\
\label{cond-a_min_b}
g_y(\xi) = f(y) \qquad \mbox{a.s.}
\end{gather}
\end{lemma}
\begin{proof}
The necessity was proved in Lemma~\ref{st-2.exxi}.
Indeed, take  $\xi$ that was obtained in the course of the proof of Lemma~\ref{st-2.exxi}.
Then condition  \eqref{cond-a_min_a} follows from the equality \eqref{umozh-1},
while condition \eqref{cond-a_min_b} follows from  the fact that $\xi \in \mathfrak{G}_{a}$.

The sufficiency is proved basically by reversing a proper separation argument from Lemma~\ref{st-2.exxi}: if a
function belongs to the convex set $\mathcal{C}$, then it cannot be properly separated from this set, which means that it is a minimizer.
To make this idea rigorous, assume the contrary: let a function $y$ satisfy \eqref{cond-a_min_a}
and \eqref{cond-a_min_b}, but $y\notin\mathfrak{M}_f$.
Then there exists function $a\in \LT$  such that  $f(y)>f(a)$
(for example, we can take   $a$ as the minimizing function).
Functional $f^2$ is convex, therefore
\[
f(y + \delta(a-y))^2 \le f(y)^2 + \delta\,(f(a)^2 - f(y)^2),
\qquad 0 \le \delta \le 1.
\]
It is easy to see that for any function  $b \in \LT $
\begin{multline*}
\max_{t\in [0,1]} \| K_y(t,\cdot) - b \|^2 =
\max_{t \in [0,1]} \biggl(
\int_0^t (K(t,s) - b(s))^2 ds +
\int_t^1 (y(s) - b(s))^2 ds \biggr) \le \\
\le \max_{t\in [0,1]}
\int_0^t (K(t,s) - b(s))^2 ds +
\int_0^1 (y(s) -b(s))^2 = f(b)^2 + \|y-b\|^2 .
\end{multline*}
Therefore for  $0 \le \delta \le 1$ we have that
\[
\max_{t\in [0,1]} \| K_y(t,\cdot) - y - \delta\,(a-y)\|^2 \le
f(y)^2 - \delta \, (f(y)^2-f(a)^2) + \delta^2 \|a-y\|^2.
\]
It means that for sufficiently small $\delta>0$
\begin{equation}
\label{eq-L813}
\max_{t\in [0,1]} \| K_y(t,\cdot) - y - \delta\,(a-y)\|^2 < f(y)^2 .
\end{equation}
On one hand, choose arbitrary $\delta$ for which the inequality \eqref{eq-L813} holds,
and set $b = y + \delta\,(a-y)$. Then
\begin{equation}
\label{eq-L819}
\max_{t\in [0,1]} \| K_y(t,\cdot) - b\|^2 < f(y)^2 .
\end{equation}
On the other hand,
\begin{align}
\label{eq-L824}
\nonumber
\max_{t\in[0,1]} \| K_y(t,\cdot)-b\|^2 &\ge
\M \|K_y(\xi,\cdot) - b\|^2 \ge
\M \|K_y(\xi,\cdot) - \M K_y(\xi,\cdot)\|^2 = \\ &=
\M \|K_y(\xi,\cdot) - y\|^2 = \M g_y(\xi)^2 = f(y)^2.
\end{align}
Inequalities  \eqref{eq-L819} and  \eqref{eq-L824} contradict each other.
So, assuming that function $y$ is not minimizing for principal functional $f$,
we get the contradiction. Therefore, ${f(y) = \min f}$.
\end{proof}

%\fussy
Now we  are in position to prove that
\[
\esssup \xi: = \min\{t : \pr(\xi\le t) = 1\} = \max(\supp \xi)=1,
\]
which will imply that $t_2=1$ in Lemma~\ref{st-2.exxi}.

\begin{lemma}\label{st-3.esssup}
Let $a$ be the minimizing function for principal functional $f$ and let
 $\xi$ be random variable satisfying conditions
\eqref{cond-a_min_a} and  \eqref{cond-a_min_b} with $x=a$.
Then $\esssup \xi = 1$.
\end{lemma}

\begin{proof}
Denote $t_2 = \esssup \xi$. Evidently,  $\xi$ takes values from $[0, t_2]$.

Consider a function
\[
b(s) = t_2^{-\alpha} a(t_2 s), \qquad s\in[0,1].
\]
Then, in view of the self-similarity property
(item 4 in Lemma~\ref{lem_K_prop}),
\[
b(s) = \M K_b(\xi/t_2,\: s),
\]
where $K_b(t, s)$ is defined in the formulation of Lemma~\ref{thm-cond_a_min}.
Using~\eqref{selfsim}, we get
\[
g_b(t) = t_2^{-H} g_a(t_2 t), \qquad t\in[0,1].
\]
On one hand, since $a(s)$ satisfies
\eqref{cond-a_min_b}, we have
\[
f(b) = \max_{[0,1]} g_b \ge g_b\bigr({\textstyle\frac{\xi}{t_2}}\bigr) =
t_2^{-H} g_a(\xi) = t_2^{-H} f(a)
\]
a.s.; on the other hand
\[
f(b) = \max_{[0,1]} g_b = t_2^{-H} \max_{[0, t_2]} g_a \le
t_2^{-H} \max_{[0,1]} g_a = t_2^{-H} f(a) .
\]
This implies
\[
f(b) = g_b\bigr({\textstyle\frac{\xi}{t_2}}\bigr) = t_2^{-H} f(a)
\qquad \text{a.e.}
\]

Therefore, the function $b$ satisfies~\eqref{cond-a_min_a} and \eqref{cond-a_min_b}
and is therefore a minimizer of $f$. Hence
\[
t_2^{-H} f(a) = f(b) = \min_{\LT} f = f(a),
\]
so $t_2 = 1$, as required.
\end{proof}

\subsection{Main properties of the minimizing function}

We can refine Lemma~\ref{st-2.exxi} in view of Lemma~\ref{st-3.esssup}.
We remind that $a$ is the minimizing function for the principal
functional  $f$ and  $\mathfrak{G}_a = \{t\in[0,1] : g_a(t) = f(a)\}$.
\begin{theorem}\label{cor-vaE}
There exists a random variable $\xi_a$ assuming values in $\mathfrak{G}_a$ such that
\begin{gather}
\pr(\xi_a \ge s) > 0 \qquad \mbox{for all $s\in[0,1)$}, \nonumber \\
a(s) = \M[K(\xi_a,s) \mid \xi\ge s] \qquad \mbox{a.e. in $[0,1]$}.\label{minimizator-form}
\end{gather}
\end{theorem}
\begin{proof}
This statement is a straightforward consequence of Lemma~\ref{st-3.esssup}.
\end{proof}
We will assume further (clearly, without loss of generality) that \eqref{minimizator-form} holds for \emph{every} $s\in[0,1]$:
\begin{equation}
a(s) = \M[K(\xi_a,s) \mid \xi\ge s] \qquad \mbox{for any $s\in[0,1]$}.\label{minimizator-form1}
\end{equation}

\begin{corollary}\label{thm-wbmf}
1. The minimizing function $a$ is left-continuous and has right limits.

2. For any $s\in[0,1)$
\begin{equation}
0 < a(s) \le K(1,s) ,
\label{neq1-thm-wbmf}
\end{equation}
moreover,
\[
a(s)< K(1,s)
\]
on a set of positive Lebesgue measure.%
%
%Знайдеться така функція $\phi : (0,1] \to \mathbb{R}$, яка задовольняє нерівність
%$s <\phi(s) \le 1$ при $s\in(0,1)$, і для якої
%\[
%a(s) = K(\phi(s), s) \quad
%\mbox{майже скрізь на $(0,1]$.}
%\]
\end{corollary}

\begin{proof}1. Follows from \eqref{minimizator-form1}, continuity of $K$ and the dominated convergence.

2. Taking into account statement 2 of Lemma~\ref{lem_K_prop}, for $0 < s < t \le 1$
\begin{equation*}
0 < K(t,s) \le K(1,s).
%\label{neq-1047}
\end{equation*}
Now \eqref{neq1-thm-wbmf} follows from \eqref{minimizator-form1} and the fact that $P(\xi_a>s)>0$ for $s<1$.
Further, if $a(s)=K(1,s)$ a.e., then $g_a(1)=0$,
which contradicts Lemma~\ref{lem_sup_point}.
%
%Існування точки $\phi(s)$ випливає з нерівності
%\eqref{neq1-thm-wbmf} з урахуванням рівності
%$K(s,s) = 0$ та неперервності $K(t,s)$.
\end{proof}

Further we investigate the distribution of $\xi$.
\begin{lemma}\label{prop_1sS}
There exists $t^* \in (0,1)$ such that
\[
\forall t \in (t^*, 1)\: : \: g_a(t) < f(a)
\]
\end{lemma}

\begin{proof}
Denote
\[
h(t) = g_a(t)^2 = \int_{0}^t (K(t,s) - a(s))^2 ds .
\]
The function $h$ is continuous on $[0, 1]$ and has
left and right derivatives (except of $h'_+(0) = +\infty$):
\begin{gather*}
h'_{-}(t) = a(t)^2 + 2\int_0^t \big(K(t,s)-a(s)\big)\,K'_t(t,s)\, ds,\\
h'_{+}(t) = a(t+)^2 + 2\int_0^t \big(K(t,s)-a(s)\big)\,K'_t(t,s)\, ds,
\end{gather*}
where $K'_t(t,s) = \frac{\partial}{\partial t} K(t,s)
= C_\alpha s^{-\alpha} t^\alpha (t-s)^{\alpha-1}$.
Hence, by Corollary \ref{thm-wbmf}
\[
h'_{-}(1) = 2\int_0^1 \big(K(1,s)-a(s)\big)\,K'_t(1,s)\, ds > 0,
\]
and the statement easily follows.
\end{proof}

The lemma just proved means that $1$ is an isolated point of  $\mathfrak{G}_a$.

As an immediate corollary, we have the following theorem.
\begin{theorem}
There exists $t^*_a<1$ such that $\pr(\xi_a\in(t^*_a,1))=0$, and the distribution of $\xi_a$ has an atom at $1$, i.e. $\pr(\xi=1)>0$. Consequently, $a(s) = K(1,s)$ for all $s\in [t^*_a,1]$.
\end{theorem}
Further we prove that the distribution of $\xi_a$ has no other atoms.
\begin{theorem}\label{thm-noatoms}
For any $t\in(0,1)$ \ $\pr(\xi_a=t)=0$. Consequently, $a\in C[0,1]$.
\end{theorem}
\begin{proof}
We start by computing for $t\in(0,1)$
\begin{equation}
\label{minimizerincrement}
\begin{gathered}
a(t+) - a(t) = \M[K(\xi_a,t)|\xi_a>t]-\M[K(\xi_a,t)|\xi_a\ge t]\\ =\frac{\M[K(\xi_a,t)1_{\xi_a>t}]\pr(\xi_a\ge t)-\M[K(\xi_a,t)1_{\xi_a\ge t}]\pr(\xi_a> t)}{\pr(\xi_a> t)\pr(\xi_a\ge t)}\\
= \frac{\M[K(\xi_a,t)1_{\xi_a>t}]\pr(\xi_a = t)-\M[K(\xi_a,t)1_{\xi_a=t}]\pr(\xi_a> t)}{\pr(\xi_a> t)\pr(\xi_a\ge t)}
\\= \frac{\M[K(\xi_a,t)1_{\xi_a>t}]\pr(\xi_a = t)}{\pr(\xi_a> t)\pr(\xi_a\ge t)}-\frac{\M[K(t,t)1_{\xi_a=t}]}{\pr(\xi_a\ge t)}
= \frac{a(t+)\pr(\xi_a=t)}{\pr(\xi_a\ge t)}.
\end{gathered}
\end{equation}
Further, as in the proof of Lemma~\ref{prop_1sS}, denote $h=g_a^2$ and observe that it has left and right derivatives at $t$ equal to
\begin{gather*}
h'_{-}(t) = a(t)^2 + 2\int_0^t \big(K(t,s)-a(s)\big)K'_t(t,s) ds,\\
h'_{+}(t+) = a(t+)^2 + 2\int_0^t \big(K(t,s)-a(s)\big)K'_t(t,s) ds.
\end{gather*}
But for any $t\in \mathfrak{G}_a$ \ $h'_{-}(t)\ge 0$, $h'_{+}(t+)\le 0$, so $a(t)\ge a(t+)$, whence from \eqref{minimizerincrement} we have that $a(t+) = a(t)$ and also  $\pr(\xi_a=t)=0$, as $a(t+)>0$.  For $t\notin \mathfrak{G}_a$ \ $\pr(\xi_a=t)=0$
(recall that $\xi_a$ takes values in $\mathfrak{G}_a$) and $a(t+)=a(t)$.
\end{proof}

\begin{remark}
Due to monotonicity of $K$ in the first variable, the right-hand of inequality \eqref{eq-sj1}
is maximal for $t_2=1$, so we have
that
\begin{equation}\label{lowerestimate}
f(a) \ge \frac 14 \max_{t\in[0,1]} \int_0^t \big(K(1,s)-K(t,s)\big)^2 ds.
\end{equation}

Theorem~\ref{thm-noatoms} implies in particular that the inequality is strict, i.e.\ this lower bound is not attained. Indeed,
if there were equality in \eqref{lowerestimate}, Lemma~\ref{lem_K_ineq} would imply that the distribution of $\xi_a$ is $\frac12(\delta_{t_0}+\delta_1)$,
where $t_0$ is the point where the minimum of the right-hand side of \eqref{lowerestimate} is attained, which would contradict Theorem~\ref{thm-noatoms}.
\end{remark}
\begin{remark}
From \eqref{minimizator-form1} it is easy to see that $a$ decreases on the complement of $\mathfrak{G}_a$. The numerical experiments in the following section suggest that $a$ is decreasing on $[0,1]$ (the positive jumps in the graphs are due to atoms, which are, clearly, unavoidable in the discrete case, but there are no atoms in the continuous) case. It seems even that $a$ is constant on $\mathfrak{G}_a\setminus\{1\}$, which would be a striking property to have. However, we did not manage to prove either of these facts.
\end{remark}

\section{Approximation of a discrete fBm by martingales}
In this section we consider a problem of minimization of the principal functional, but in discrete time. This is an approximation to the original problem, so its solution can be considered as an approximate solution to the original problem.

Let $N$ be a natural number, and define $b_k = B^H_{k/N}, k = 0,\ldots, N$. %--- значення дробового броунівського руху в точках $0, 1/N, \ldots, N/N$.
The vector  $b=(b_0,b_1,\ldots,b_N)$ will be called a discrete fBm. It generates a discrete filtration
$\F_k=\sigma(b_0,\ldots,b_k)$, $k=0,\ldots, N$. For arbitrary  random vector  $\xi = (\xi_0, \xi_1,\ldots,\xi_N)$ with square
integrable components denote
$$G(\xi) = \max_{k=0,\ldots,N}\M(b_k-\xi_k)^2.$$

Consider the problem of minimization of the functional  $G(\xi)$, where
$\xi$ is an $\F_k$-martingale.

Denote by $d_i = b_i-b_{i-1}, i=1,\ldots,N$ the increments of the discrete
fBm. Let $C$ be the covariance matrix of the vector  $(d_i|i=1,\ldots,N)$.
Using the Cholesky decomposition, one can find a lower triangular real matrix $L=(l_{ij}|i,j=1,\ldots,N)$
such that
$C=LL^T$. Then there exists a sequence $(\zeta_1,\ldots,\zeta_N)$ of independent standard
Gaussian random variables such that  $\zeta_k$ is $\F_k$-measurable for
$k=1,\ldots,N$ and
\begin{equation*}%\label{eq_chol1}
\begin{pmatrix}
d_1\\
\vdots\\
d_N
\end{pmatrix} =
L \begin{pmatrix}
\zeta_1\\
\vdots\\
\zeta_N
\end{pmatrix}.
\end{equation*}

Define a matrix $K=(k_{ij}|i,j=1,\ldots,N)$ as follows:
$$
k_{ij} = \begin{cases}
            0,& i < j \\
            \sum_{s=1}^i l_{sj} & i \ge j.
         \end{cases}
$$
%\begin{lemma}
%Має місце рівність
It is clear that
$$
\begin{pmatrix}
b_1\\
\vdots\\
b_N
\end{pmatrix} =
K \begin{pmatrix}
\zeta_1\\
\vdots\\
\zeta_N
\end{pmatrix}.
$$
%\end{lemma}
%\begin{proof}
%Випливає з $b_s=\sum_{i=1}^s d_i$, рівності \eqref{eq_chol1} та
%визначення матриці $K$.
%\end{proof}
The matrix $K$ is therefore can be regarded as a discrete counterpart of a
fractional Brownian kernel.

Further, we will show, as in the continuous case, that  minimization of $G$ over martingales is equivalent to minimization
over Gaussian martingales. Indeed, let $\xi = \xi=(\xi_0 = 0,\xi_1,\ldots,\xi_N)$ be arbitrary square integrable $\F_k$-martingale. Owing to the fact that $\F_k=\sigma\{\zeta_1,\dots,\zeta_k\}$, $k=1,\dots,N$, we have the following martingale representation:
$$
\xi_n = \sum_{k=1}^n\alpha_k \zeta_k,\quad n=1,\dots,N,
$$
where $\alpha_k$ is a square integrable $\F_k$-measurable random variable, $k=1,\dots,N$. Thus,
\begin{gather*}
G(\xi) = \max_{j=0,\ldots,N}\M(b_j-\xi_j)^2 = \max_{j=0,\ldots,N} \sum_{n=1}^j \M(k_{j\,n} - \alpha_n )^2 \\
= \max_{j=0,\ldots,N} \sum_{n=1}^j \big(\M(k_{j\,n} - \M \alpha_n )^2+\mathrm{Var}(\alpha_n)\big)\ge \max_{j=0,\ldots,N} \sum_{n=1}^j \M(k_{j\,n} - \M \alpha_n )^2.
\end{gather*}
So we can assume that $\xi$ has a form $\xi_k=\sum_{j=1}^k
a_j\zeta_j$, $k=1,\dots,N$, with some non-random $a_1,\dots,a_n$. Then
\begin{equation*}
G(\xi)= \max_{t=1,\ldots,N} \sum_{s=1}^t(k_{ts}-a_s)^2 =: F(a).
\end{equation*}

Thus, we have arrived to the following optimization problem:
$$
\min F(a),\qquad a\in\R^N.
$$

For fixed  $N$ and $H$ we solve this problem numerically by using the  MATLAB \texttt{fminimax} function.

The following table gives the values of the functional for different
$H$ and $N=200$.
\begin{table}[h]
\begin{tabular}{|c|c|c|c|c|c|c|c|c|c|}
\hline
H & .55 & .6 & .65 & .7 & .75 & .8 & .85 & .9 & .95\\
\hline
$\min F(a)$ & .0013  & .0051    & .0112  & .0200   & .0320    & .0482    & .0705   & .1023   & .1511 \\
\hline
\end{tabular}
\end{table}

Figure~\ref{fig:valminf} shows the values of  $\min F$ for $H$ from
$0.51$ to $0.99$ with a step $0.01$ for $N=200$. Figure~\ref{fig:minimizer} contain graphs of the minimizing vector (blue) and the scaled ``distance'' $R(t)=\sum_{s=1}^t(k_{ts}-a_s)^2$ (red), when $H = 0.75$ and $N = 500$. For other values of $H$ the picture
is similar: $a$ is (mainly) decreasing and looks close to constant on the sets of maxima of $R$.

\begin{figure}[htb]
\begin{center}
\includegraphics[width=.9\textwidth]{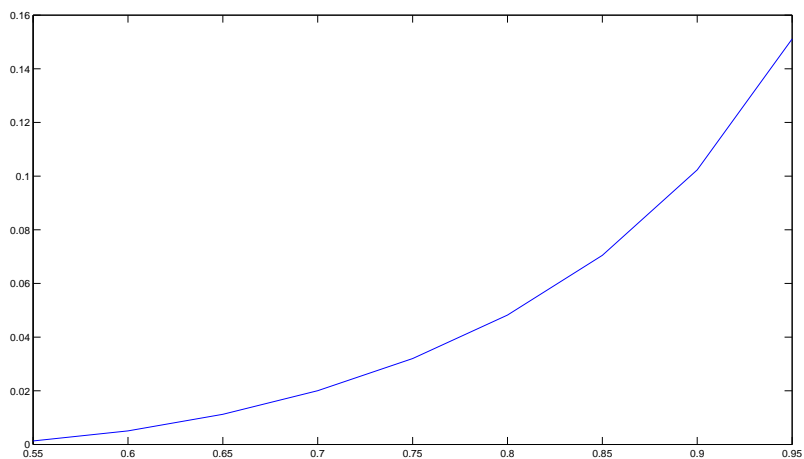}
\end{center}
\caption{Values of  $\min F$ for $H$ from
$0.51$ to $0.99$.}\label{fig:valminf}
\end{figure}

% (вісь $X$ -
%номер координати, вісь $Y$ - значення координати).
\begin{figure}[htb]
\begin{center}
\includegraphics[width=.9\textwidth]{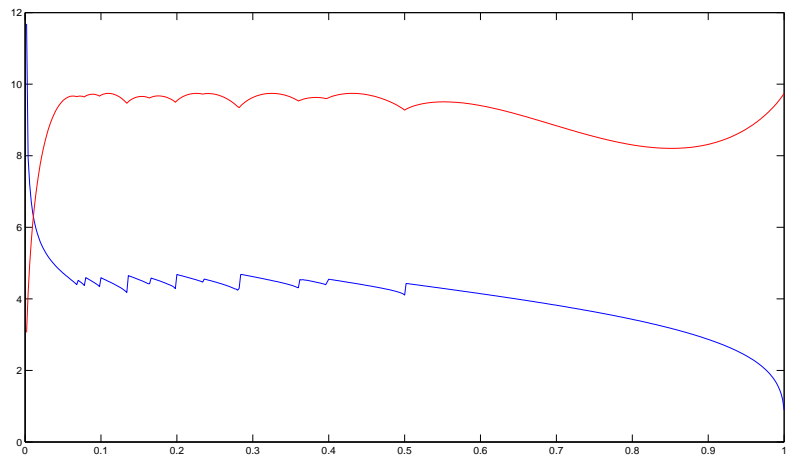}
\end{center}
\caption{The minimizing vector (blue) and the scaled distance (red) for $H = 0.75$}\label{fig:minimizer}
\end{figure}

\end{document}